\documentclass[12pt]{article}

\usepackage{amsmath}
\usepackage{amsfonts}
\usepackage{amssymb}
\usepackage{amsthm}
\usepackage{showkeys}
\usepackage{mathrsfs}
\usepackage{hyperref}

\newcommand{\bb}[1]{\mathbb{#1}}
\newcommand{\Z}[1]{\bb{Z}_{#1}}
\newcommand{\F}[1]{\bb{F}_{#1}}

\newcommand{\floor}[1]{\lfloor #1 \rfloor}

\newcommand{\qbinom}[3]{\left[ #1 \atop #2 \right]_{\! #3}}

\DeclareMathOperator{\GL}{GL}

\DeclareMathOperator{\M}{M}

\DeclareMathOperator{\Irr}{Irr}

\DeclareMathOperator{\Gal}{Gal}

\DeclareMathOperator{\Fix}{Fix}

\DeclareMathOperator{\ppd}{ppd}

\newtheorem{thm}{Theorem}[section]
\newtheorem{cor}[thm]{Corollary}

\newtheorem{lem}[thm]{Lemma}
\newtheorem{proposition}[thm]{Proposition}

\theoremstyle{definition}
\newtheorem{defn}[thm]{Definition}

\newtheorem{remark}[thm]{Remark}

\title{Nilpotent-independent sets and estimation in matrix algebras}
\author{Brian Corr\footnote{School of Mathematics and Statistics, 
The University of Western Australia. Current address: Departamento de Matem\'{a}tica, 
Instituto de Ci\^{e}ncias Exatas, 
Universidade Federal de Minas Gerais, 
Av.~Ant\^{o}nio Carlos, 6627, 31270-901
Belo Horizonte, MG, Brazil; brian.p.corr@gmail.com.}, 
Tomasz Popiel\footnote{School of Mathematics and Statistics, 
The University of Western Australia, 
Australia; tomasz.popiel@uwa.edu.au.}, 
Cheryl E. Praeger\footnote{School of Mathematics and Statistics, 
The University of Western Australia, Australia, and King Abdullaziz University, Jeddah, Saudi Arabia; cheryl.praeger@uwa.edu.au.}}

\begin{document}
\maketitle

\begin{abstract}
Efficient methods for computing with matrices over finite fields often involve {\em randomised} algorithms, 
where matrices with a certain property are sought via repeated random selection. 
Complexity analyses for these algorithms require knowledge of the proportion of relevant matrices in the ambient group or algebra. 
We introduce a method for estimating proportions of families $N$ of elements in the algebra of all $d\times d$ matrices over a field of order $q$, where membership of a matrix in $N$ depends only on its `invertible part'. 
The method is based on estimating proportions of certain subsets of $\GL(d,q)$ depending on $N$, so that existing estimation techniques for nonsingular matrices can be leveraged to deal with families containing singular matrices. 
As an application we investigate primary cyclic matrices, which are used in the Holt--Rees \texttt{MEAT-AXE} algorithm for testing irreducibility of matrix algebras. 
\end{abstract}

\section{Introduction}

In order to develop efficient methods for computing with matrices over finite fields, it is often necessary to use randomised algorithms as opposed to deterministic algorithms: the latter are often too slow because the size of the group or algebra grows exponentially with the size of the input. 
Indeed, most algorithms for computing in finite matrix groups or algebras are either Monte Carlo or Las Vegas algorithms, both of which have a small user-controlled probability of error or failure as a caveat to being far more efficient than corresponding deterministic algorithms. 
(A Monte Carlo algorithm is guaranteed to terminate but its output may be incorrect with small probability; a Las Vegas algorithm may fail to terminate with small probability but is otherwise guaranteed to return a correct output.)

Randomised algorithms typically rely on a randomised search for certain `desirable' matrices: there will be some theoretical result justifying the correctness of the algorithm which says that if a certain kind of matrix can be found, then the question being considered can be resolved. 
For example, the Neumann--Praeger~\cite{neumann1992recognition} and Niemeyer--Praeger~\cite{niemeyer1998recognition} algorithms for recognising finite classical groups in their natural representations rely on finding elements with orders divisible by certain primes, while the Holt--Rees version of the \texttt{MEATAXE} algorithm~\cite{HoltRees} for testing irreducibility of a finite matrix group or algebra utilises primary cyclic matrices. 
Complexity analyses of such algorithms therefore depend on estimating the number of desirable elements in the given group or algebra.
Various methods are used to solve such estimation problems, depending on their exact nature. 
For example, Glasby and Praeger~\cite{GlasbyPraegerfcyclic} use a generating function approach to estimate the proportion of primary cyclic matrices arising in the \texttt{MEATAXE} algorithm~\cite{HoltRees}. 

The {\em quokka theory} of Niemeyer and Praeger~\cite{niemeyer2010estimating} is an algebraic group-theoretic method for estimating the cardinality of subsets $Q$ of finite simple groups of Lie type such that $Q$ is a union of conjugacy classes and membership of $Q$ depends only on the semisimple part of the Jordan decomposition of an element. 
This technique is similar to one used by Lehrer~\cite{lehrer1992rational,lehrer1998cohomology} to study representations of finite Lie type groups and has recently proven useful for several estimation problems~\cite{lubeck2009finding, niemeyer2010proportions, niemeyer2013abundant}.
In the present paper we aim to extend the quokka theory in a certain sense to the full matrix algebra $M = \text{M}(d,q)$. 
By analogy, we deal with subsets $N$ of $M$ for which inclusion depends only on the {\em nilpotent} part of the matrix. 
The technique itself involves estimating the cardinality of certain subsets $N_i$ of $\text{GL}(i,q)$ ($1\leq i\leq d$) related to $N$, and therefore allows one to utilise existing methods (such as quokka theory) that apply only to nonsingular matrices in order to treat families containing singular matrices. 
This research forms part of the first author's Ph.D. thesis \cite[Chapter~6]{BrianThesis}.

Our formula for the estimating the size of a nilpotent-independent set is presented in Section~\ref{sec2} (Theorem~\ref{sum}), where we also discuss an application to primary cyclic matrices and the \texttt{MEATAXE} algorithm (Theorem~\ref{PC in M}). 
The proofs of Theorems~\ref{sum} and~\ref{PC in M} are given in Sections~\ref{sec3} and~\ref{sec4}, respectively.

\subsection{Definitions and main results} \label{sec2}

Let $V=\mathbb{F}_q^d$ be the $d$-dimensional space of row vectors over the field $\mathbb{F}_q$, and let $\M(V) = M(d,q)$ be the algebra of linear transformations of $V$. 
Our main theorem relates the size of a subset $N$ of $\M(V)$ satisfying certain properties to the sizes of certain subsets $N_i$ of $\GL(i,q)$, $1\leq i\leq d$, that are determined by $N$ together with a fixed maximal flag of $V$ (see Definition~\ref{maximal flag Defn}). 
Each $X\in \M(V)$ determines a unique decomposition
\[
V = V_\text{inv}(X) \oplus V_\text{nil}(X)
\]
such that $X_\text{inv}:= X|_{V_\text{inv}(X)}$ is invertible and $X_\text{nil}:= X|_{V_\text{nil}(X)}$ is nilpotent. 
We call $X_\text{inv}$ the \emph{invertible part} and $X_\text{nil}$ the \emph{nilpotent part} of $X$, and we write $X = X_\text{inv} \oplus X_\text{nil}$. 
In the language of primary decompositions~\cite{hartleyhawkes}, $V_\text{nil}(X)$ is precisely the $t$-primary component of $V$ and $V_\text{inv}(X)$ is the direct sum of all the other primary components; that is, $V_\text{inv}(X) = \oplus_{f\in\Irr(q), f\neq t} V_f(X)$, where $\Irr(q)$ denotes the set of monic irreducible polynomials in $\mathbb{F}_q[t]$.

\begin{defn}\label{nice}
A subset $N$ of $\M(V)$ is called a \emph{nilpotent-independent} (NI) subset if the following conditions hold:
\begin{enumerate}
\item $N$ is closed under conjugation by elements of $\GL(V)$, and
\item for $X\in \M(V)$, we have $X\in N$ if and only if $X_\text{inv}\oplus 0_{V_\text{nil}(X)}\in N$, where $0_{V_\text{nil}(X)}$ is the zero transformation on $V_\text{nil}(X)$.
\end{enumerate}
\end{defn}

In the same sense that membership of Niemeyer and Praeger's {\em quokka sets} \cite{niemeyer2010estimating} (see Section~\ref{3.2}) depends only on the semisimple part of the Jordan decomposition of $g\in \GL(V)$, condition (ii) above says that membership of an NI subset depends only on the invertible part of $X\in \M(V)$, and is independent of the nilpotent part. 
In particular, unions of conjugacy classes of $\GL(V)$ are NI subsets: for a nonsingular matrix $X$, $X_\text{nil}=0$ and hence condition (i) above holds vacuosly for all families of nonsingular matrices. 
Therefore, all quokka subsets of $\GL(V)$ are NI subsets.

\begin{defn}\label{maximal flag Defn}
A \emph{maximal flag} of $V$ is a family of suspaces $V_1,\ldots,V_d$ such that $\{0\}=V_0 \subset V_1 \subset \cdots \subset  V_d = V$.
Note that $\dim V_i = i$ for $0\leq i\leq d$. Given a maximal flag $\{V_i\}$ and an NI subset $N$, we write, for each $i$,
\begin{align*}
N(i) &= \{X\in N \mid \dim(V_\text{inv}(X)) = i \}, \\
N_i &= \{Y\in\GL(V_i) \mid Y=X_\text{inv} \text{ for some } X\in N \text{ such that } V_\text{inv}(X)=V_i\}.
\end{align*}
The set $\{N_i \mid 0\leq i\leq d\}$ is called the \emph{NI family corresponding to $N$ and $\{V_i\}$}.
\end{defn}

Note that, since $N$ is closed under conjugation, the $N(i)$ do not depend on the maximal flag $\{V_i\}$ (but the $N_i$ do depend on $\{V_i\}$). 
Also, fixing a maximal flag is a weaker condition that fixing an ordered basis since an ordered basis $\{ v_1,\ldots,v_d \}$ determines the maximal flag $\{ V_i \}$ with $V_i = \langle v_1,\ldots,v_i \rangle$ for $i\geq 1$. 

We are interested in NI subsets that contain noninvertible elements. 
Each such set determines (up to conjugacy in $\GL(V)$) a collection of sets of invertible elements in smaller dimensions, namely the $N_i$ above. 
In Section~\ref{sec2} we derive the following precise relationship between the size of $N$ and the sizes of the $N_i$, thus reducing the enumeration problem in $\M(d,q)$ to a set of enumeration problems in $\GL(i,q)$, $0\leq i \leq d$.

\begin{thm}\label{sum}
Let $\{V_i \mid 0\leq i\leq d\}$ be a maximal flag of $V=\mathbb{F}_q^d$ and let $N$ be an NI subset of $\M(V)$. Then each $N_i$ is a union of conjugacy classes of $\GL(V_i)$, the family $\{N_i \mid 0\leq i\leq d\}$ as in Definition~$\ref{maximal flag Defn}$ is unique up to $\GL(V)$-conjugacy, and 
\begin{equation}\label{formula}
\frac{|N|}{|\GL(V)|} = \sum_{i=0}^{d} \frac{q^{-(d-i)}}{\omega(d-i,q)} \frac{|N_i|}{|\GL(V_i)|},
\end{equation}
where $\omega(0,q) = 1$ and $\omega(j,q) = \prod_{k=1}^j (1-q^{-k}) = |\GL(j,q)|/|\M(j,q)|$, $j\geq 1$.
\end{thm}

\begin{remark}
The proportion $|N|/|\M(V)|$ is, of course, obtained from \eqref{formula} upon multiplying by 
$\omega(d,q) = |\GL(V)|/|\M(V)|$. 
\end{remark}

Many interesting subsets of $\M(V)$ are nilpotent-independent, including any set for which membership is determined by the structure of the characteristic or minimal polynomial (see Lemma~\ref{CP Quokka}). 
In particular, the set of primary cyclic matrices, namely those whose characteristic polynomial and minimal polynomial share an irreducible factor with the same multiplicity, is an NI subset of $\M(V)$. 
In Section~\ref{sec4} we apply Theorem~\ref{sum} to obtain a lower bound on the proportion of matrices in $\M(V)=\M(c,q^b)$ that are primary cyclic when viewed as elements of a larger, ambient matrix algebra $M(bc,q)$ which contains $\M(c,q^b)$ as an irreducible (but not absolutely irreducible) subalgebra. 
Specifically, we prove the following result. 

\begin{thm}\label{PC in M}
Let $b,c \geq 2$ be integers and let $N$ be the set of matrices $X$ in $\M(c,q^b)\subseteq \M(bc,q)$ that are primary cyclic with respect to some irreducible polynomial $f(t) \neq t$ of degree greater than $\dim(V_\textnormal{inv}(X))/2$. Then 
\[
\frac{|N(c,q,b)|}{|\M(c,q^b)|}  > \log 2 - \frac{\log 2 + 3}{c} - \frac{2(1-1/c)}{q^{b/2}}.
\]
\end{thm}

\begin{remark} \label{ppd}
The set $N$ in Theorem~\ref{PC in M} contains the set $P$ of so-called {\em primitive prime divisor} elements of $\GL(c,q^b)$, namely nonsingular matrices $X$ with order divisible by a prime that divides $q^{bi}-1$ for some $i>c/2$ but does not divide $q^j-1$ for any $j<bi$. 
The proportion $|P|/|\GL(c,q^b)|$ is approximately $\log 2$ \cite[Theorem~6.1]{niemeyer1998recognition}, and it seems reasonable that $|N|/|\GL(c,q^b)|$ should also be roughly $\log 2$. 
Theorem~\ref{PC in M} shows that this is the case for even modest values of $b,q$. 
\end{remark}

\begin{remark}
Testing irreducibility with the Holt--Rees \texttt{MEATAXE} algorithm~\cite{HoltRees} uses primary cyclic matrices obtained by random selection from an algebra $M$. 
A lower bound on the proportion of primary cyclic matrices in $M$ is needed to justify that the algorithm is a Monte Carlo algorithm and to determine its complexity. 
For the case where $M$ is a full matrix algebra $\M(V)$, such lower bounds were given by Holt and Rees~\cite{HoltRees} and improved upon by Glasby and Praeger~\cite{GlasbyPraegerfcyclic}. 
In the case where $M$ is a proper irreducible subalgebra of $\M(V)$, namely the case considered in this paper, Theorem~\ref{PC in M} gives an explicit lower bound for the proportion of matrices that are primary cyclic with respect to a polynomial of large degree. 
By contrast, the first and third authors~\cite{CorrPraegerPC1} have previously determined a lower bound on the proportion of matrices that are primary cyclic with respect to an irreducible polynomial of smallest possible degree. 
\end{remark}

\section{Nilpotent-independent subsets} \label{sec3}

In this section we prove Theorem~\ref{sum} and then deduce some corollaries that give bounds on the cardinality of $N$ under certain generic assumptions.

\subsection{Proof of Theorem~\ref{sum}}

We begin with a lemma about the structural relationship between the sets $N(i)$ and $N_i$ in Definition~\ref{maximal flag Defn}.

\begin{lem}\label{mainlemma}
Let $N$ be an NI subset of $\M(V)$, $V=\mathbb{F}_q^d$, let $\{V_i \mid 0\leq i\leq d\}$ be a maximal flag of $V$, and for $0\leq i\leq d$ define $N_i, N(i)$ as in Definition~$\ref{maximal flag Defn}$. Then the following hold:
\begin{enumerate}
\item For each $i$, $N_i$ is closed under $\GL(V_i)$-conjugacy.
\item The set $N_0 \subseteq \GL(V_0)$ is empty if $N$ contains no nilpotent elements, and has size $1$ otherwise.
\item For a maximal flag $\{V_i' \mid 0\leq i\leq d\}$ with corresponding NI family $\{N_i' \mid 0\leq i\leq d\}$, there exists $g\in \GL(V)$ such that, for each $i$, $V_i^g = V_i'$ and $N_i^g=N_i'$.
\item For each $i$, $|N(i)| = \qbinom{d}{i}{q} q^{(d-i)(d-1)}|N_i|$, where 
\[
\qbinom{d}{i}{q} = \frac{|\GL(d,q)|}{|\GL(i,q)||\GL(d-i,q)|} q^{-i(d-i)}
\]
is the $q$-binomial coefficient, namely the number of $i$-dimensional subspaces of $V$.
\end{enumerate}
\end{lem}
\begin{proof}
(i) If $N_i$ is empty then there is nothing to prove, so suppose that $N_i$ is nonempty and let $X_i\in N_i$. 
Then there exists $X\in N$ with $V_\text{inv}(X) = V_i$, $X_\text{inv}=X_i$ and $X_\text{nil}=0_{V_\text{nil}(X)}$. 
Now let $x\in \GL(V_i)$. Then $x'=x \oplus I_{V_\text{nil}(X)} \in \GL(V)$, where $I_{V_\text{nil}(X)}$ is the identity map on $V_\text{nil}(X)$. 
Since $N$ is closed under conjugacy, $X^{x'} = X_i^x \oplus 0_{V_\text{nil}(X)}\in N$. 
Hence $(X^{x'})_\text{inv} = X_i^x$ is the invertible part of the element $X^{x'}$ of $N$ and it lies in $\GL(V_i)$, so $X_i^x\in N_i$. Thus $N_i$ is closed under conjugacy.

(ii) If $N$ contains no nilpotent elements then there is no $X\in N$ with $\dim V_\text{inv}(X) = 0$, and hence $N_0$ is empty. 
If $N$ contains a nilpotent element $X$, then $V_\text{inv}(X) = \{0\} = V_0$ and $X_\text{inv}$, the identity map on $V_0$, lies in $N_0$.

(iii) Let $\{v_i \mid 1\leq i \leq d\}, \{v_i' \mid 1\leq i \leq d\}$ be bases for $V$ such that, for $1\leq i\leq d$, the sets $\{v_j \mid 1\leq j\leq i\}, \{v_j' \mid 1\leq j\leq i\}$ are bases for $V_i$, $V_i'$ respectively. 
Then the transformation $g\in\GL(V)$ defined by $v_i^g = v_i'$, $1\leq i\leq d$, and extended by linearity to $V$ has the desired properties.
%

(iv) Write $N(V_i) = \{X \in N \mid V_\text{inv}(X) = V_i \}$. Let $X_i\in N_i$. Then for every complement $U$ of $V_i$ in $V$, and for every nilpotent $n\in \M(U)$, we have $X_i \oplus n \in N(V_i)$. Moreover, each different choice of $U,n$ yields a different element of $N(V_i)$, and all of $N(V_i)$ arises in this way. Thus the size of $N(V_i)$ is precisely $|N_i|$ times the number $q^{i(d-i)}$ of complements $U$, times the number $q^{(d-i)(d-i-1)}$ of nilpotent elements in $\M(U)$ \cite{gerstenhaber1961number}. 
The set $N(i)$ is the disjoint union of $N(V_i')$ over all $i$-dimensional subspaces $V_i'$ of $V$. By (ii) and (iii), all of the $N(V_i')$ have the same size $|N(V_i)|$, and so $|N(i)|$ is equal to $|N(V_i)|$ times the number of $i$-dimensional subspaces of $V$.
The result follows.
\end{proof}


Let us now prove Theorem~$\ref{sum}$. Recall that we want to show that
\[
\frac{|N|}{|\GL(V)|} = \sum_{i=0}^{d} \frac{q^{-(d-i)}}{\omega(d-i,q)} \frac{|N_i|}{|\GL(V_i)|}.
\]

\begin{proof}[Proof of Theorem~$\ref{sum}$] 
The first assertions are proved in Lemma~\ref{mainlemma}. It remains to prove \eqref{formula}. 
Note that $|\GL(d-i,q)| = q^{(d-i)^2}\omega(d-i,q)$ for all $i$. 
Lemma~\ref{mainlemma} gives
\begin{align*}
\frac{|N(i)|}{|\GL(d,q)|} 
&= \frac{1}{|\GL(d,q)|} \qbinom{d}{i}{q} q^{(d-i)(d-1)} |N_i| \\
&= \frac{1}{|\GL(d,q)|} \left(\frac{|\GL(d,q)|}{|\GL(i,q)||\GL(d-i,q)|} q^{-i(d-i)} \right) q^{(d-i)(d-1)} |N_i| \\
&= \frac{q^{(d-i)(d-i-1)}}{|\GL(d-i,q)|} \frac{|N_i|}{|\GL(i,q)|} \\
&= \frac{q^{-(d-i)}}{\omega(d-i,q)} \frac{|N_i|}{|\GL(i,q)|}.
\end{align*}
Since the $N(i)$ partition $N$, $|N|=\sum_{1\leq i\leq d} |N(i)|$ and the result follows. 
\end{proof}

It is unusual when enumerating sets in $\GL(V)$ to consider $0$-dimensional cases, but the $0$th term of the sum in (\ref{formula}) is well behaved:

\begin{remark}
By definition, an NI subset $N$ of $\M(V)$ must contain either all nilpotent elements of $\M(V)$, or none. In the former case, the $0$th term of (\ref{formula}) is
\[
\frac{q^{-d}}{\omega(d,q)} = q^{-d} \frac{|\M(V)|}{|\GL(V)|}.
\]
In the latter case, the $0$th term is $0$.
\end{remark}

\subsection{Some generic lower bounds for $|N|$}

If we can estimate each proportion $|N_i|/|\GL(i,q)|$ in terms of $i$ and $q$ then we can use (\ref{formula}) to estimate the proportion $|N|/|\M(d,q)|$. 
In this way, estimation techniques that are normally effective only in $\GL(d,q)$ (for example, quokka theory) can be used to deal with subsets of $\M(d,q)$. 
If we can find bounds on the $|N_i|/|\GL(i,q)|$ that behave `uniformly' in some sense, for example, as in Proposition~\ref{exptransfer} or Proposition~\ref{lineartransfer}, then (\ref{formula}) can be applied without much additional effort. 
We first prove a useful formula by considering the case $N=\M(d,q)$.

\begin{cor}\label{sum2}
For any prime power $q$ and any positive integer $d$,
\begin{equation}\label{NI Corollary Sum}
\sum_{i=0}^d \frac{q^{-(d-i)}}{\omega(d-i,q)} = \sum_{i=0}^{d} \frac{q^{-i}}{\omega(i,q)} = \frac{1}{\omega(d,q)}.
\end{equation}
Equivalently,
\begin{equation}\label{sum2eqn}
\sum_{i=1}^d \frac{q^{-(d-i)}}{\omega(d-i,q)} = \sum_{i=0}^d \frac{q^{-(d-i)}}{\omega(d-i,q)} - \frac{q^{-d}}{\omega(d,q)} = \frac{1-q^{-d}}{\omega(d,q)}.
\end{equation}
\end{cor}
\begin{proof}
The first equality in \eqref{NI Corollary Sum} is just a change of variable. Now consider $N=\M(d,q)$. 
Then $N$ is an NI Subset and, for every $i$, $N_i = \GL(i,q)$. By Theorem~\ref{sum},
\[
\frac{|N|}{|\GL(d,q)|} = \sum_{i=0}^d \frac{q^{-(d-i)}}{\omega(d-i,q)} \cdot 1
\]
and so the left-hand side of (\ref{NI Corollary Sum}) is equal to $|M(d,q)|/|\GL(d,q)|$, 
which is $1/\omega(d,q)$.
\end{proof}

\begin{proposition}\label{exptransfer}
Let $d$ be a positive integer, $N$ an NI subset of $V=\mathbb{F}_q^d$ and $\{N_i\}$ a corresponding NI family. 
Suppose that there exist constants $a,k > 0$ such that $|N_i|/|\GL(i,q)| \geq a - k q^{-i}$ for $1\leq i\leq d$. 
Then
\[
 \frac{|N|}{|\M(d,q)|} \geq a - (a+k)dq^{-d}\geq a - (a+k)\left(\frac{2q}{3}\right)^{-d}.
 \]
\end{proposition}

\begin{proof}
Applying \eqref{formula} and \eqref{sum2eqn} and, we find 
\begin{align*}
\frac{|N|}{|\M(d,q)|} &= \omega(d,q) \frac{|N|}{|\GL(d,q)|} 
 =  \omega(d,q)\left( \sum_{i=0}^{d} \displaystyle\frac{q^{-(d-i)}}{\omega(d-i,q)}. \frac{|N_i|}{|\GL(V_i)|}\right) \\
&\geq  \omega(d,q)\left( 0 + \sum_{i=1}^{d} \displaystyle\frac{q^{-(d-i)}}{\omega(d-i,q)}. (a - kq^{-i})\right) \\
&= a \omega(d,q) \sum_{i=1}^{d} \displaystyle\frac{q^{-(d-i)}}{\omega(d-i,q)} - k \omega(d,q) q^{-d} \sum_{i=1}^{d} \frac{1}{\omega(d-i,q)},
\end{align*}
and using \eqref{sum2eqn} this is equal to 
$a (1-q^{-d}) - k \omega(d,q) q^{-d} \sum_{i=1}^{d} 1/\omega(d-i,q)$. 
Noting that $\omega(d-i,q) \geq \omega(d-1,q) = \omega(d,q)/(1-q^{-d})$ for $1\leq i \leq d$, this is at least 
$a (1-q^{-d}) - k (1-q^{-d}) d {q^{-d}} \geq  a - (k+a) d {q^{-d}}$. 
Since $d < (3/2)^d$ for all integer values of $d$,
\[
(a+k) dq^{-d} < (a+k) \left(\frac{3}{2}\right)^d q^{-d}=(a+k) \left(\frac{2q}{3}\right)^{-d},
\]
and the second asserted inequality follows. 
\end{proof}

A similar result holds when we have slower convergence to the limiting proportion. 
We need the following lemma, which is easily verified.

\begin{lem}\label{sumbound}
For all $d\geq 1$ and $q \geq 2$,
\[
d\sum_{i=1}^d \frac{q^i}{i} < 3q^d.
\]
\end{lem}


\begin{proposition}\label{lineartransfer}
Let $d$ be a positive integer, $N$ be an NI subset of $V=\mathbb{F}_q^d$ and $\{N_i\}$ a corresponding NI family. 
Suppose that $|N_i|/|\GL(i,q)| \geq a - k/i$ for $1\leq i\leq d$ for some $a,k>0$. Then 
\[
\frac{|N|}{|\M(d,q)|} \geq \left(a-\frac{3k}{d}\right)(1-q^{-d})  > a - \frac{a+3k}{d}.
\]
\end{proposition}
\begin{proof}
Applying (\ref{formula}) and using the assumed bounds and the fact that $|N_0| \geq 0$, 
\begin{align*}
\frac{|N|}{|\M(d,q)|} 
&\geq  \omega(d,q) \sum_{i=1}^{d} \displaystyle\frac{q^{-(d-i)}}{\omega(d-i,q)} \left(a - \frac{k}{i} \right) \\
&= a \omega(d,q) \sum_{i=1}^{d} \displaystyle\frac{q^{-(d-i)}}{\omega(d-i,q)} - k \omega(d,q) \sum_{i=1}^{d} \frac{q^{-(d-i)}}{i\omega(d-i,q)}\\
&= a (1-q^{-d}) - k \omega(d,q) q^{-d} \sum_{i=1}^{d} \displaystyle\frac{q^i}{i\omega(d-i,q)},
\end{align*}
where we use \eqref{sum2eqn} for the last equality. As $\omega(d-i,q) \geq \omega(d-1,q)$ for every $i$ considered,
\[
\frac{|N|}{|\M(d,q)|} \geq   a (1-q^{-d}) - k (1-q^{-d}) q^{-d} \sum_{i=1}^{d} \frac{q^i}{i},
\]
which by Lemma \ref{sumbound} is greater than 
$a (1-q^{-d}) - k (1-q^{-d}) q^{-d} \cdot 3q^d/d = (a - 3k/d)(1-q^{-d})$. 
The result follows since $d <q^d$ for all $d \geq 1$, giving
\[
\left(a - \frac{3k}{d}\right)(1-q^{-d}) > a - \frac{3k}{d} - \frac{a}{q^d} > a - \frac{3k}{d} - \frac{a}{d} .
\]
\end{proof}

\section{An application to primary cyclic matrices}\label{sec4}

Recall that a matrix $X\in \M(n,q)$ is \emph{primary cyclic} if there exists a monic irreducible polynomial $f\in \mathbb{F}_q[t]$ such that the multiplicities of $f$ in the characteristic polynomial $c_{X,V(n,q)}(t)$ and minimal polynomial $m_{X,V(n,q)}(t)$ are equal and at least $1$. 
Here we use the notation $c_{X, V(n,q)}(t),m_{X, V(n,q)}(t)$ to denote the characteristic and minimal polynomials of $X$ {\em in its action on} $V(n,q)$: this is necessitated by our consideration of actions over different fields. 
This is equivalent to the requirement that the action of $X$ on its $f$-primary component is cyclic. For a discussion of primary cyclic matrices and their significance (they are used in the Holt--Rees \texttt{MEATAXE} algorithm, central to recognition of matrix groups), we refer the reader to Glasby~\cite{glasby2006meat} and Corr and Praeger~\cite{CorrPraegerPC1}.

In this section we use quokka theory to determine lower bounds on the proportion of primary cyclic matrices in a subgroup $\GL(c,q^b)$ of $\GL(bc,q)$, and apply our theory of NI subsets to obtain a lower bound on the proportion of primary cyclic matrices in an irreducible subalgebra $\M(c,q^b)$ of $\M(bc,q)$.

\subsection{Primary cyclic matrices in $\M(c,q^b)$}

For $X\in \M(c,q^b) \subset \M(bc,q)$, we write $X_{c,q^b}$ and $X_{bc,q}$ for the unique linear transformations of $V(c,q^b)$ and $V(bc,q)$ induced by $X$, respectively. 
That is, $X_{c,q^b}$ acts on a $c$-dimensional $K$-vector space, where $K = \mathbb{F}_{q^b}$; and $X_{bc,q}$ acts on a $bc$-dimensional $F$-vector space, where $F = \mathbb{F}_q$. 
A key result is Proposition~\ref{polynomials}, proved in \cite{CorrPraegerPC1}, which gives necessary and sufficient conditions for a matrix $X\in \M(c,q^b)$ to be primary cyclic when viewed as an element of the larger algebra $\M(bc,q)$ (that is, for $X_{bc,q}$ to be primary cyclic). 
This characterisation involves the Galois group $\Gal(K/F)$ of automorphisms of $K$ fixing $F$ pointwise. 
As before, $\Irr(q)$ denotes the set of monic irreducible polynomials in $F[t]$, and $\Irr_m(q)$ denotes the subset of degree $m$ polynomials in $\Irr(q)$. 

\begin{proposition}\label{polynomials}
Let $f\in \Irr(q)$ and $X\in\M(c,q^b)$ such that $f$ divides $c_{X,V(bc,q)}(t)$. 
Then $X_{bc,q}$ is $f$-primary cyclic if and only if $b$ divides $\deg(f)$ and the following hold for some divisor $g\in K[t]$ of $f$ of degree $\deg(f)/b$:
\begin{enumerate}
\item $X_{c,q^b}$ is $g$-primary cyclic, and
\item for every nontrivial $\tau\in \Gal(K/F)$, the image $g^{\tau}\neq g$ and $g^{\tau}$ does not divide $c_{X,V(c,q^b)}(t)$.
\end{enumerate}
\end{proposition}

\begin{lem}\label{Candidate Polys}
Let $r > 1$. Then each $f\in \Irr_{br}(q)$ is a product $\prod_{\tau \in \Gal(K/F)} g^\tau$, where $g\in \Irr_r(q^b)$ is such that $g^\tau \neq g$ for all nontrivial $\tau \in \Gal(K/F)$. 
In particular, the number of $g\in \Irr_r(q^b)$ with this property is $r|\Irr_{br}(q)|$. 
\end{lem}

\begin{proof}
Write $L = \mathbb{F}_{q^{br}}$. Then each $f\in \Irr_{br}(q)$ is of the form
\[
f(t) = \prod_{i=0}^{br-1} (t-\lambda^{q^i}) \quad \text{for some } \lambda \in L.
\]
For each $j \in \{1,\ldots,b\}$, define
\[
g_j(t) = \prod_{i=0}^r (t-\lambda^{q^{(i-1)b + j}}).
\]
Denote by $\sigma$ the automorphism of $L$ that raises elements to their $q$th power. 
Then for $1\leq j\leq b-1$ we have $g_j^{\sigma} = g_{j+1}$, and $g_b^{\sigma}=g_1$. 
It follows that, for each $j$, $g_j^{\sigma^b} = g_j$ and hence $g_j\in K[t]$. 
Moreover, for $f$ to be irreducible we require both that the $g_j$ should be irreducible and that they should be pairwise distinct. 
Note that $\Gal(K/F)$ consists of the restrictions $\sigma^i|_K$ for $0 \leq i < b$ (since $\sigma^b|_K=1$). 
Thus each $f\in\Irr_{br}(q)$ gives rise to exactly $b$ monic irreducible divisors $g\in K[t]$ satisfying the condition that $g^\tau \neq g$ for $1 \neq \tau \in \Gal(K/F)$. 
Moreover, for any $g$ satisfying this condition, we have $\prod_{\tau\in \Gal(K/F)} g^{\tau} \in \Irr_{br}(q)$, and so there is a bijection between $\Gal(K/F)$-orbits of length $b$ of irreducible polynomials of degree $r$ over $K$ and irreducible polynomials $f$ of degree $br$ over $F$.
\end{proof}

\begin{defn}\label{pc sets defn}
For $r,b,c\in \mathbb{Z}^+$, $q$ a prime power and $f\in \Irr(q)$, define
\begin{align*}
N(c,q,b;f) &:= \{X\in \GL(c,q^b) \mid X_{bc,q}\text{ is $f$-primary cyclic} \}, \\
N(c,q,b,r) &:= \cup_{f\in\Irr_{br}(q)} N(c,q,b;f), \\
N &:= N(c,q,b) = \cup_{r>c/2} N(c,q,b,r).
\end{align*}
\end{defn}

Note that if $b=1$ then $N(c,q,1;f)$ is the set of $f$-primary cyclic matrices in $\M(c,q)$.

Suppose that $f \in \Irr_{br}(q)$ with $r>c/2$, and that $f$ divides $c_{X,V(bc,q)}(t)$. 
Since $r>c/2$, $f$ is the only degree $br$ divisor of $c_{X,V(bc,q)}(t)$. 
Suppose also that $g \in \Irr_r(q^b)$ divides $f$ and $c_{X,V(c,q^b)}(t)$. 
Then, again since $r>c/2$, no $g^\tau \neq g$ (for $\tau\in \Gal(K/F)$) can divide $c_{X,V(c,q^b)}(t)$. Thus
\begin{itemize}
\item[(a)] $X_{c,q^b}$ is $g$-primary cyclic if and only if $X_{bc,q}$ is $f$-primary cyclic, and
\item[(b)] the sets $N(c,q,b;f)$ are pairwise disjoint for $f \in \cup_{r>c/2}\Irr_{br}(q)$. 
\end{itemize}
In particular, $N(c,q,b)$ is a subset of the set of primary cyclic matrices in $\M(bc,q)$ lying in $\M(c,q^b)$, and so a lower bound for $|N|$ gives a lower bound for the number of primary cyclic matrices $X_{bc,q}$ in $\M(c,q^b)$.

Our goal is to determine the size of $N(c,q,b,r)$ for fixed $r>c/2$, by first enumerating $N(c,q,b;f)$ for a fixed $f$ satisfying certain conditions. 
We use the approach described in Section~\ref{3.2} to estimate the cardinality of these sets.

\subsection{Quokka theory} \label{3.2}

In order to derive upper and lower bounds for the size of $N(c,q,b;f) \subseteq GL(c,q^b)$ as in Definition~\ref{pc sets defn}, we apply the theory of {\em quokka sets} of $G=\GL(n,q)$ \cite{lubeck2009finding,niemeyer2010estimating} (the theory can be applied to all finite groups of Lie type, but here we need only the linear case). 
These are subsets whose proportion in $G$ can be determined by considering certain proportions in maximal tori in $G$ and certain proportions in the corresponding Weyl group. 
Recall that each element $g\in G$ has a unique Jordan decomposition $g =  su$, where $s\in G$ is semisimple, $u\in G$ is unipotent and $su=us$ (with $s$ called the \emph{semisimple part} of $g$ and $u$ the \emph{unipotent part}) \cite[p.~11]{carter1993finite}. 
Note that the order $o(s)$ of $s$ is coprime to the characteristic of $G$, and that $o(u)$ is a power of the characteristic.

As per~\cite[Definition~1.1]{niemeyer2010estimating}, a nonempty subset $Q$ of $G$ is called a \emph{quokka set} if the following two conditions hold:
\begin{itemize}
\item[(i)] If $g\in G$ has Jordan decomposition $g=su$ with semisimple part $s$ and unipotent part $u$, then $g\in Q$ if and only if $s \in Q$.
\item[(ii)] $Q$ is a union of $G$-conjugacy classes.
\end{itemize}
We note again the analogy with the definition of an NI subset of $\M(n,q)$. Indeed, the latter was formulated as a way to extend quokka theory to $\M(n,q)$.

Let $\bar{\mathbb{F}}_q$ denote the algebraic closure of $\F{q}$, with $\phi$ the Frobenius morphism (so that the fixed points of $\phi$ in $\bar{\mathbb{F}}_q$ are precisely the elements of $\F{q}$). 
As outlined in~\cite[Section 3]{lubeck2009finding}, choose a maximal torus $T_0$ of $ \GL(n,\bar{\F{q}})$ so that $W = N_{\hat{G}}(T_0)/T_0$ is the corresponding Weyl group, and note that for the linear case $W$ is isomorphic to $S_n$. 
We summarise the results about quokka subsets of $G$ that are used in the proof of Proposition~\ref{Qr Prop}.
A subgroup $H$ of the connected reductive algebraic group $\GL(n,\bar{\mathbb{F}}_q)$ is said to be \emph{$\phi$-stable} if $\phi(H)=H$, and for each such subgroup $H$ we write $H^\phi=H\cap \GL(n,\F{q})$. 
Define an equivalence relation on $W$ as follows: elements $w,w'\in W$ are \emph{$\phi$-conjugate} if there exists $x\in W$ such that  $w' =x^{-1} w x^\phi$. 
The equivalence classes of this relation on $W$ are called \emph{$\phi$-conjugacy classes} \cite[p.~84]{carter1993finite}. 
The $\GL(n,\F{q})$-conjugacy classes of $\phi$-stable maximal tori are in one-to-one correspondence with the $\phi$-conjugacy classes of the Weyl group $W \cong S_n$. 
The explicit correspondence is given in~\cite[Proposition 3.3.3]{carter1993finite}.

Let $\mathcal{C}$ be the set of $\phi$-conjugacy classes in $W$ and, for each $C\in\mathcal{C}$, let $T_C$ be a representative element of the family of $\phi$-stable maximal tori corresponding to $C$. 
The following theorem is a direct consequence of \cite[Theorem 1.3]{niemeyer2010estimating}.

\begin{thm}\label{the:quokka}
Suppose that $Q \subseteq G=\GL(n,q)$ is a quokka set. Then, with the above notation,
\begin{equation}\label{QuokkaEqn}
\frac{|Q|}{|G|} =  \sum_{C \in \mathcal{C}}\frac{|C|}{|W|} \frac{|T_C^\phi\cap  Q|}{|T_C^\phi|}.
\end{equation}
\end{thm}

In order to apply Theorem~\ref{the:quokka}, we check that the sets $N(c,q^b,1;f)$ in Definition~\ref{pc sets defn} are quokka sets. 
To do this, we prove a more general statement about sets defined by properties of the characteristic polynomial.

\begin{lem}\label{CP Quokka}
Let $g\in\GL(V)$ and suppose that $g$ has multiplicative Jordan decomposition $g=su=us$, where $u$ is unipotent and $s$ is semisimple. Then $c_{g}(t) = c_s(t)$.
\end{lem}

\begin{proof}
Let $f\in\Irr(q)$ divide $c_g(t)$ with multiplicity $m$, and let $V_f = \ker(f^m(g))$ be the $f$-primary component of $g$. 
Then both $u$ and $s$ fix $V_f$ setwise, since they commute. Since $u|_{V_f} \in GL(V_f)$ is unipotent, its fixed-point space $U = \Fix u|_{V_f}$ is nontrivial. 
Now, for any $v\in U$, we have $(v^{s})^u = v^{us} = v^s$, and so $s$ fixes $U$ setwise. 
It follows that $g$ fixes $U$ setwise, and indeed $g|_U = u|_U s|_U = s_U$, that is, $s$ and $g$ agree on $U$. Hence $f^m$ divides the characteristic polynomial of $s$.
Since this holds for all $f$, it follows that $c_g(t)$ divides $c_s(t)$, and since these are both monic polynomials of the same degree, equality holds.
%
\end{proof}

\begin{remark}\label{CP properties give quokka sets}
A consequence of Lemma~\ref{CP Quokka} is that any subset of $\GL(V)$ defined by properties of its members' characteristic polynomials is a quokka set. Indeed, if membership of a subset depends only on the characteristic polynomial of $X\in \GL(V)$, then membership depends only on a property of the semisimple part of $X$. Since the characteristic polynomial is invariant under $\GL(V)$-conjugacy, it follows that sets defined in this way are quokka sets.
There are many examples of sets defined in this way, including the separable matrices, the unipotent matrices, matrices with a given eigenvalue, and the sets $N(c,q,b,r)$ of Definition~\ref{pc sets defn} for $r > c/2$, as we now prove in Lemma~\ref{g divides c is enough}.
\end{remark}

\begin{lem}\label{g divides c is enough}
Let $c,b \in \mathbb{Z}^+$, $q$ a prime power and $K=\mathbb{F}_{q^b}$, $F=\mathbb{F}_q$ as before. 
Let $r>c/2$ and let $g\in \Irr_r(q)$ satisfy $g^\tau \neq g$ for all nontrivial $\tau \in \Gal(K/F)$. 
Then, for $f=\prod_{\tau \in \Gal(K/F)} g^\tau$, we have $f\in \Irr_{br}(q)$ and $N(c,q,b;f)$ is a quokka set. 
In particular, $X \in N(c,q,b;f)$ if and only if $g^\tau$ divides $c_{X,V(c,q^b)}(t)$ for exactly one $\tau \in \Gal(K/F)$. 
\end{lem}

\begin{proof}
By hypothesis all the $g^\tau$, $\tau \in \Gal(K/F)$, are distinct and hence $f\in \Irr(q)$ with $\operatorname{deg}(f) = br$. 
Suppose that $X \in \M(c,q^b)$ is such that some $g^\tau$ divides $c_{X,V(c,q^b)}(t)$. 
Then, since $r>c/2$, it is not possible for $g^{\tau'}$ to divide $c_{X,V(c,q^b)}(t)$ for any $\tau'\neq\tau$, and also $(g^\tau)^2$ cannot divide $c_{X,V(c,q^b)}(t)$. 
Hence $X_{c,q^b}(t)$ is $g^\tau$-primary cyclic, and it follows from Proposition~\ref{polynomials} that $X_{bc,q}$ is $f$-primary cyclic. So $X\in N(c,q,b;f)$. 
Conversely, if $X\in N(c,q,b;f)$ then by Proposition~\ref{polynomials}, $X_{c,q^b}$ is $g^\tau$-primary cyclic and hence $g^\tau$ divides $c_{X,V(c,q^b)}(t)$ for exactly one $\tau \in \Gal(K/F)$. 

Since conjugate matrices have the same characteristic polynomial, condition (ii) for a quokka set holds. 
Condition (i) also holds, for suppose that $X\in N(c,q,b;f)$ with Jordan decomposition $X=US=SU$. 
We have just proved that $g^\tau$ divides $c_{X,V(c,q^b)}(t)$ for exactly one $\tau \in \Gal(K/F)$. 
Let $W$ be its $g^\tau$-primary component in $V(c,q^b)$. 
Then $X|_W$ is irreducible and as $U,S$ centralise $X$, they both leave $W$ invariant and both $U|_W, S|_W$ centralise $X|_W$. 
Since $U|_W$ is unipotent, it follows that $U|_W=1$ and hence $X|_W=S|_W$, which implies that $g^\tau$ divides $c_{S,V(c,q^b)}(t)$. 
Thus, arguing as above, $\tau$ is unique with this property and $S\in N(c,q,b;f)$. So $N(c,q,b;f)$ is a quokka set.
\end{proof}

\begin{cor}\label{Quokka For pc}
With notation as in Lemma~$\ref{g divides c is enough}$,
\[
\frac{|N(c,q,b;f)|}{|\GL(c,q^b)|} = \frac{b}{q^{br}-1}.
\]
\end{cor}
\begin{proof}
Since $Q := N(c,q,b;f)$ is a quokka set, the required proportion is given by \eqref{QuokkaEqn}. 
Now, $T_C\cap Q$ is nonempty if and only if $T_C$ contains an element $X\in Q$ or equivalently, by Lemma~$\ref{g divides c is enough}$, $g^\tau$ divides $c_{X,V(c,q^b)}(t)$. 
This implies that all permutations in $C \subset W \cong S_c$ contain an $r$-cycle, and conversely, for all such $C$, $T_C\cap Q$ is nonempty. Each such torus $T_C$ has the form 
\[
\Z{q^{br}-1} \times S,
\]
where $S$ corresponds to parts outside the $r$-cycle. That is, one of the components of the torus $T_C$ is the multiplicative group of a field extension $\F{q^{br}}$: precisely $r$ elements of this field are roots of $g^\tau$ and so precisely $r$ elements of the corresponding torus factor $\Z{q^{br}-1}$ have characteristic polynomial $g^\tau$ on this subspace $K^r$. 
This is true for each $\tau \in \Gal(K/F)$. Thus
\[
\frac{|N(c,q,b;f) \cap T_C|}{|T_C|} = \frac{br}{q^{br}-1}.
\]
Hence, if $\mathcal{C}'$ denotes the classes of $S_c$ containing an $r$-cycle, then
\[
\frac{|N(c,q,b;f)|}{|\GL(c,q^b)|} 
= \sum_{C\in\mathcal{C}'} \frac{|C|}{|S_c|} \frac{br}{q^{br}-1}
= \left(\sum_{C\in\mathcal{C}'} \frac{|C|}{|S_c|}\right) \frac{br}{q^{br}-1}
= \frac{1}{r} \frac{br}{q^{br}-1}
\]
since the proportion of permutations containing an $r$-cycle is $1/r$. 
\end{proof}


\begin{proposition}\label{Qr Prop}
For $c,b,r \in \mathbb{Z}^+$ with $r>c/2$, and $q$ a prime power, 
\[
\frac{|N(c,q,b,r)|}{|\GL(c,q^b)|} = \frac{b|\Irr_{br}(q)|}{q^{br}-1}.
\]
In particular, 
\[
\frac{1}{r}(1 - 2q^{-br/2}) < \frac{|N(c,q,b,r)|}{|\GL(c,q^b)|}\leq \frac{1}{r}. 
\]
\end{proposition}

\begin{proof}
Since $r>c/2$, $N(c,q,b,r)$ is the disjoint union of the sets $N(c,q,b;f)$ for $f\in \Irr_{br}(q)$. 
Thus, by Corollary~\ref{Quokka For pc}, the first assertion holds. For the bounds, note that
\begin{equation}\label{BoundOnNEqn}
 \frac{1}{br}(q^{br} - 2q^{br/2}) \leq |\Irr_{br}(q)| \leq \frac{q^{br}-1}{br},
\end{equation}
for in the proof of Lemma~\ref{Candidate Polys}, each $f\in \Irr_{br}(q)$ is a product $\prod_{i=0}^{br-1} (t-\lambda^{q^i})$ for some $\lambda \in \mathbb{F}_{q^{br}}$ lying in no proper subfield containing $F$, and by \cite[Lemma~4.2]{niemeyer2013abundant} there are at least $q^{br}-2q^{br/2}$ such elements $\lambda$. 

The first inequality in \eqref{BoundOnNEqn} gives
\begin{align*}
\frac{b|\Irr_{br}(q)|}{q^{br}-1}
&\geq \frac{b}{q^{br}-1} \frac{1}{br}(q^{br} - 2q^{br/2})  \\
&= \frac{q^{br}(1-2q^{-br/2})}{r(q^{br}-1)} > \frac{1-2q^{-br/2}}{r},
\end{align*}
since $1-2q^{-br/2} \geq 0$.
\end{proof}

As Proposition \ref{Qr Prop} demonstrates, the proportion $|N(c,q,b,r)|/|\GL(c,q^b)|$ is approximately $1/r$. 
We use this to derive estimates for $|\cup_{r > c/2} N(c,q,b,r)|$. 
The following lemma is easily verified and we omit the proof for brevity. 

\begin{lem}\label{Sum 1/r}
Let $c \geq 2$. Then
\[
\log 2 - \frac{1}{c+1} \leq  \sum_{r=\floor{\frac{c}{2} +1}}^c \frac{1}{r}  \leq \log 2 + \frac{1}{c}.
\]
\end{lem}


\begin{proposition}\label{Q Prop in GL}
For $N(c,q,b)$ as in Definition~$\ref{pc sets defn}$,
\[
\log 2 - \frac{1}{c+1} - \frac{2}{q^{bc/4}}
< 
\frac{|N(c,q,b)|}{|\GL(c,q^b)|} 
\leq 
\log 2 +  \frac{1}{c}. 
\]
\end{proposition}
\begin{proof}
By definition $N(c,q,b)=\cup_{r>c/2} N(c,q,b,r)$, and the $N(c,q,b,r)$ are pairwise disjoint, because no two polynomials of degree greater than $c/2$ can divide the characteristic polynomial of any one matrix. Thus
\[
\frac{|N(c,q,b)|}{|\GL(c,q^b)|} 
= \sum_{r>c/2} \frac{|N(c,q,b,r)|}{|\GL(c,q^b)|} 
\]
and so, by Proposition~\ref{Qr Prop},
\[
\sum_{r=\floor{c/2}+1}^c \frac{1}{r}  (1- 2q^{-br/2})
\leq 
\frac{|N(c,q,b)|}{|\GL(c,q^b)|} 
\leq  
\sum_{r=\floor{c/2}+1}^c \frac{1}{r}.
\]
The asserted upper bound for $|N(c,q,b)|/|\GL(c,q^b)|$ now follows from Lemma \ref{Sum 1/r}. 
For the lower bound, first apply Lemma \ref{Sum 1/r} to get
\[
\frac{|N(c,q,b)|}{|\GL(c,q^b)|}  \geq \log 2 - \frac{1}{c+1} - \sum_{r=\floor{c/2}+1}^c \frac{2}{rq^{-br/2}}.
\]
To bound the remaining sum, observe that there are $\lceil c/2 \rceil$ summands with
\[
-\frac{2}{rq^{-br/2}} \geq -\frac{2}{r_0q^{-br_0/2}}, \quad \text{where } r_0 := \lfloor c/2 \rfloor + 1.
\]
For $c$ even this yields
\[
- \sum_{r=\floor{c/2}+1}^c \frac{2}{rq^{-br/2}} \geq -\frac{2\cdot c/2}{(c/2+1)q^{bc/4}} 
> -\frac{2}{q^{bc/4}}, 
\]
and for $c$ odd
\[
- \sum_{r=\floor{c/2}+1}^c \frac{2}{rq^{-br/2}} \geq -\frac{2\cdot (c+1)/2}{(c+1)/2 \cdot q^{bc/4}} 
= -\frac{2}{q^{bc/4}}. 
\]
\end{proof}

\begin{remark}\label{Comparison to PPD elements remark}
The bounds in Proposition \ref{Q Prop in GL} are similar to the bounds obtained by Niemeyer \& Praeger \cite[Theorem 6.1]{niemeyer1998recognition} on the proportion $P$ of elements $g\in \GL(c,q)$, $c\geq 3$, such that $g$ is a so-called $\ppd(c,q;r)$-element for some $r > c/2$. 
This means that the order of $g$ is divisible by a primitive prime divisor (ppd) of $q^r-1$, namely a prime that divides $q^r-1$ but does not divide $q^j-1$ for any $j<r$ (as per Remark~\ref{ppd}). The proportion $P$ satisfies
\[
\log 2 - \frac{1}{c+2} \leq P \leq \log 2 + \frac{1}{c-1}.
\]
This kind of result, with linear convergence to the limit, seems to be the best that can be obtained by considering polynomials of large degree. 
We note that the set $N(c,q,b)$ is both more and less restrictive than the set of ppd elements. 
On the one hand, some matrices in $N(c,q,b)$ may have order not divisible by a ppd of $q^r-1$; on the other hand, some ppd elements correspond to irreducible polynomials $g\in K[t]$ that do not have the property $g^\tau \neq g$ for nontrivial $\tau \in \Gal(K/F)$. 
Thus the two sets are very similar but neither is contained in the other. 
\end{remark}

In order to apply Theorem \ref{sum} to prove Theorem~\ref{PC in M}, we first note that Lemmas \ref{exptransfer} and \ref{lineartransfer} rely on knowledge of the proportion $|N_i|/|\GL(i,q)|$ for \emph{all} values of $i$. In defining the nilpotent-independent set that we wish to investigate, we must take care when considering matrices $X\in \M(d,q)$ with $\dim (V_\text{inv}(X)) \leq 2$.

%
%


\begin{proof}[Proof of Theorem~$\ref{PC in M}$]
Let $N \subset M(c,q^b)$ be as in Theorem~\ref{PC in M}.
Choose a maximal flag $\{0\} = V_0 \subset V_1 \subset \ldots \subset V_c = V(c,q^b)$ 
with $\operatorname{dim} V_i = i$ as an $\mathbb{F}_{q^b}$-space, and define $N(i)$ and $N_i$ 
as in Definition~\ref{maximal flag Defn}, where we interpret $V_\text{inv}(X)$ as an 
$\mathbb{F}_{q^b}$-space, for $X \in N$. 
Then by Theorem~\ref{sum} applied to $N$ as a subset of $M(c,q^b)$, 
\begin{equation} \label{new}
\frac{|N|}{|\GL(c,q^b)|} = \sum_{i=o}^c \frac{q^{-b(c-i)}}{\omega(c-i,q^b)} \frac{|N_i|}{|\GL(V_i)|}.
\end{equation}
Note that $N_0$ is the empty set and that $N_1 = \GL(V_1)$. 
For $i\geq 2$, $N_i$ is the subset $N(i,q,b)$ of Definition~\ref{pc sets defn} 
(with the parameter $c$ there replaced by $i$), and so, by Proposition~\ref{Q Prop in GL}, 
\[
\frac{|N_i|}{|\GL(i,q^b)|} \geq \log 2 - \frac{1}{i+1} - \frac{2}{q^{bi/4}} \geq \log 2 - \frac{1}{i+1} - \frac{2}{q^{b/2}}.
\]
This inequality also holds for $i=1$ because $|N_1|/|\GL(1,q^b)| = 1$. 
So by Proposition~\ref{lineartransfer} with $a= \log 2 - 2/q^{b/2}$ and $k=1$,  
\begin{align*}
\frac{|N(c,q,b)|}{|\M(c,q^b)|}  
&\geq \log 2 - \frac{2}{q^{b/2}} - \frac{\log 2 - 2q^{-b/2} + 3}{c} \\
&= \log 2 - \frac{\log 2 + 3}{c} - \frac{2(1-1/c)}{q^{b/2}}
\end{align*}
\end{proof}


\section*{Acknowledgements} 
This paper forms part of the first author's Ph.D. thesis at The University of Western Australia. 
He was supported by an Australian Postgraduate Award, a UWA Top-Up Scholariship and, during the writing of the paper, by an Australian Mathematical Society Lift-Off Fellowship. 

The research forms part of Australian Research Council Discovery Projects DP110101153 and DP140100416. 

We thank Stephen Glasby for several helpful discussions.

\end{document}